\documentclass[12pt, reqno]{amsart}

\usepackage{url}
\usepackage{hyperref}
\usepackage{pdfsync}
\usepackage{amssymb}
\usepackage{amsmath}
\usepackage{amsthm}
\usepackage{stmaryrd}
\usepackage{fullpage}	
\usepackage{amscd}   	
\usepackage[all]{xy} 		
\usepackage{comment} 	
\usepackage{mathrsfs}      
\usepackage{mathtools}

\usepackage{enumerate}

\usepackage[alphabetic,backrefs,lite]{amsrefs} 

\usepackage[dvipsnames]{xcolor}

\newcommand{\marg}[1]{}

\newcommand{\note}[1]{} 				

\newcommand{\defi}[1]{\textsf{#1}} 				

\newcommand{\Q}{{\mathbb Q}}

\newcommand{\Z}{{\mathbb Z}}

\newcommand{\calM}{{\mathcal M}}

\newcommand{\scrX}{{\mathscr X}}

\def\Q{\mathbb{Q}}

\def\Z{\mathbb{Z}}

\DeclareMathOperator{\im}{im}

\DeclareMathOperator{\Jac}{Jac} 

 \DeclareMathOperator{\Hom}{Hom}

\DeclareMathOperator{\Div}{Div} \DeclareMathOperator{\Pic}{Pic}

 \numberwithin{equation}{subsection}

\newtheorem{lemma}[subsection]{Lemma}

\newtheorem{thmx}{Theorem}

\newtheorem{corox}[thmx]{Corollary}

\theoremstyle{definition}
\newtheorem{definition}[subsection]{Definition}
\newtheorem{question}[subsection]{Question}

\theoremstyle{remark}
\newtheorem{remark}[subsection]{Remark}

\usepackage{tikz,tikz-3dplot}

\begin{document}
\title[Chip-firing groups of iterated cones]
{Chip-firing groups of iterated cones}
\author{Morgan V. Brown}
\address{Department of Mathematics, University of Miami,
Coral Gables, FL 33146 USA}
\email{mvbrown@math.miami.edu}
\author{Jackson S.~Morrow}
\address{Department of Mathematics and Computer Science, Emory University,
Atlanta, GA 30322 USA}
\email{jmorrow4692@gmail.com}
\author{David Zureick-Brown}
\address{Department of Mathematics and Computer Science, Emory University,
Atlanta, GA 30322 USA}
\email{dzb@mathcs.emory.edu}

\date{\today}
\thanks{}

\begin{abstract}

Let $\Gamma$ be a finite graph and let $\Gamma_n$ be the ``$n$th cone over $\Gamma$'' (i.e., the join of $\Gamma$ and the complete graph $K_n$). We study the asymptotic structure of the chip-firing group $\Pic^0(\Gamma_n)$.

\end{abstract}

\maketitle

\section{Introduction}
\label{sec:introduction}
The \emph{chip-firing} groups $\Pic^0(\Gamma) \subset \Pic(\Gamma)$ of a finite graph $\Gamma$ are classical objects of combinatorial study. Baker \cite{Baker:specialization} developed the connection between line bundles on a semistable arithmetic curve $\scrX$ and $\Pic^0(\Gamma)$, where $\Gamma$ is the dual graph of the special fiber of $\scrX$, and with various coauthors \cite{BakerN:RR, BakerN:Harmonic} discovered that the cornerstone theorems satisfied by algebraic curves (e.g., Riemann--Roch and Clifford's theorem) admit non-trivial analogous theorems for graphs. 

The technology transfer flows both ways; chip-firing (and variants and tools from tropical geometry) have emerged as a central tool in recent results across several subfields of algebraic/arithmetic geometry and number theory, including the maximal rank conjecture for quadrics \cite{jensenP:tropical-independence-II}, the  Gieseker--Petri theorem \cite{jensenP:tropical-independence-I}, the Brill--Noether theorem \cite{Cools2012}, and the uniform boundedness conjecture \cite{KatzRZB-uniform-bounds}; see \cite{bakerJ:degeneration-survey} for an extensive survey. 

An interest in the \emph{computational} properties of $\Pic^0(\Gamma)$ has recently emerged. Several authors, including \cite{BergetMM:critical-group-line-graph, JacobsonNR:critical-groups-multipartite, Lorenzini:smith-normal-form}, have worked to compute $\Pic^0(\Gamma)$ (or, failing that, $|\Pic^0(\Gamma)|$, which is equal to the number of spanning trees of $\Gamma$ \cite[Theorem~6.2]{BakerS:chip-firing-matrix-tree}) for various families of graphs; we refer the reader to \cite[pg.~1155]{alfaro2012sandpile} for nearly a complete list of authors contributing to this area. 

Our question of interest is the behavior of the chip-firing group of the $n$th cone $\Gamma_n$ over $\Gamma$, where $\Gamma_n$ is defined as the join of $\Gamma$ with the complete graph $K_n$. Recall, the \defi{join} of two graphs $\Gamma_1$ and $\Gamma_2$ is a graph obtained from $\Gamma_1$ and $\Gamma_2$ by joining each vertex of $\Gamma_1$ to all vertices of $\Gamma_2$. 
In \cite{alfaro2012sandpile}, the authors interpret the chip-firing group of the $n$th cone of the Cartesian product of graphs as a function of the chip-firing group of the cone of their factors. As a consequence, they completely describe the chip-firing group of the $n$th cone over the $d$-dimensional hypercube.
\\

Our main theorem concerns the the chip-firing group of the $n$th cone over a fixed graph. 

\begin{thmx}
\label{thm:cone}
Let $\Gamma$ be a graph on $k\geq 1$ vertices. Let $n \geq 1$ be an integer, and let $\Gamma_n$ be the $n$th cone over $\Gamma$ defined above. Then there is a short exact sequence of abelian groups 
\[
0 \to (\mathbb{Z}/(n+k)\mathbb{Z})^{n-1} \to \Pic^0(\Gamma_n) \to H_n \to 0
\]
where the order of $H_n$ is $|P_{\Gamma}(-n)|$ and $P_{\Gamma}(x)$ is the characteristic polynomial of the rational Laplacian operator. 
\end{thmx}

In particular, this immediately gives an exact formula for the number of spanning trees of $\Gamma_n$.

\begin{corox}\label{coro:size}
Let $\Gamma$ be a graph on $k\geq 1$ vertices. Let $n \geq 1$ be an integer, and let $\Gamma_n$ be the $n$th cone over $\Gamma$ defined above.  There is a subgroup of $\Pic^0(\Gamma_n)$ isomorphic to $(\mathbb{Z}/(n+k)\mathbb{Z})^{n-1}$, and 
$$|\Pic^0(\Gamma_n)| = (n+k)^{n-1}|P_{\Gamma}(-n)|$$
where $P_{\Gamma}(x)$ is the characteristic polynomial of the rational Laplacian operator. 
\end{corox}

\begin{remark}
In the statements of Theorem \ref{thm:cone} and Corollary \ref{coro:size}, the rational Laplacian operator is a linear endomorphism of $\Div^0(\Gamma)\otimes \mathbb{Q}$.  We refer the reader to Section \ref{sec:notation} for more details.
\end{remark}

\begin{remark}
In a previous version of this paper, the authors erroneously claimed that this exact sequence was split for odd values of $n+k$, and conjectured it was split in general. We are very grateful to Gopal Goel for pointing out this error and providing a counter example. If $\Gamma$ is the graph given in Figure \ref{fig:Goelexample}, a computer calculation shows that $\Pic^0(\Gamma_3) \cong \mathbb{Z}/9\mathbb{Z} \oplus \mathbb{Z}/27\mathbb{Z} \oplus (\mathbb{Z}/16\mathbb{Z})^2 \oplus \mathbb{Z}/19\mathbb{Z}$. In particular, the map $(\mathbb{Z}/9\mathbb{Z})^{2} \to  \Pic^0(\Gamma_3)$ cannot be split.

\begin{figure}[h!]
\begin{tikzpicture}[scale=1][line join=bevel]
\coordinate (0) at (-2,0);
\coordinate (1) at (-1,1);
\coordinate (2) at (-1,-1);
\coordinate (3) at (1,1);
\coordinate (4) at (1,-1);
\coordinate (5) at (2,0);

\filldraw [black] (0) circle (3pt)
(1) circle (3pt)
(2) circle (3pt)
(3) circle (3pt)
(4) circle (3pt)
(5) circle (3pt);


\draw [ultra thick] (0)--(1);
\draw [ultra thick] (0)--(2);
\draw [ultra thick] (1)--(2);
\draw [ultra thick] (1)--(3);
\draw [ultra thick] (1)--(4);
\draw [ultra thick] (2)--(3);
\draw [ultra thick] (2)--(4);
\draw [ultra thick] (3)--(4);
\draw [ultra thick] (4)--(5);

\draw [ultra thick] (3)--(5);

\end{tikzpicture} 
\caption{Goel's example: For the third cone $\Gamma_3$, the exact sequence of Theorem \ref{thm:cone} is not split.}
\label{fig:Goelexample}
\end{figure}
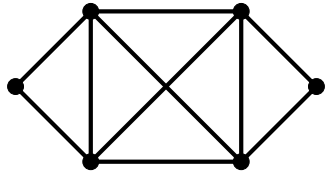

\end{remark}

A more elusive question is the precise structure of the groups $\Pic^0(\Gamma_n)$ and $H_n$. 

\begin{question}
For which $\Gamma$ and $n$, is the short exact sequence in Theorem \ref{thm:cone} split?
  \end{question}

If $\Gamma$ is a tree, then we are able to determine an upper bound on the number generators for the subgroup $H_n$ appearing above. Recall that a \defi{leaf} of a graph is a vertex of degree 1.

\begin{thmx}
\label{thm:tree}
Let $\Gamma$ be a tree with $l+1 \geq 2$ leaves and let $\Gamma_n$ and $H_n$ be as in Theorem \ref{thm:cone}.   Then, $H_n$ can be generated by $l$ elements.
\end{thmx}

It is possible that $H_n$ may be generated by fewer elements, as in Figure \ref{fig:generators}.  
\\
\begin{figure}[h!]
\begin{tikzpicture}[scale=1.2][line join=bevel]
\coordinate (0) at (-2,0);
\coordinate (1) at (-1,0);
\coordinate (2) at (0,0);
\coordinate (3) at (1,1);
\coordinate (4) at (1,-1);

\filldraw [black] (0) circle (2pt) 
(1) circle (2pt)
(2) circle (2pt)
(3) circle (2pt)
(4) circle (2pt);


\draw [ultra thick] (0)--(1);
\draw [ultra thick] (1)--(2);
\draw [ultra thick] (2)--(3);
\draw [ultra thick] (2)--(4);\end{tikzpicture}
\caption{This tree $\Gamma$ has 3 leaves, but a computer calculation shows that $\Pic^0 (\Gamma_1)$ is isomorphic to $\Z/52\Z$.} \label{fig:generators}
\end{figure}
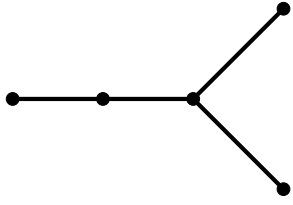

Finally,  we can slightly generalize Corollary \ref{coro:size}; we determine the order of chip-firing group for the join of $l$ graphs. 

\begin{thmx}
\label{thm:bipartite}
Let $\Gamma_1, \ldots, \Gamma_l$ be non-empty graphs with $k_1,\ldots, k_l$ vertices, and let $\Gamma_J$ be the join.
Then, 
$$|\Pic^0(\Gamma_J)| = k^{l-2}\prod_{i\leq l}|P_{\Gamma_i}(k_i-k)|,$$ where $k  = \sum k_i $ is the number of vertices of $\Gamma_J$.
\end{thmx}

Corollary \ref{coro:size} immediately follows from Theorem \ref{thm:bipartite} by taking $\Gamma_1 = \Gamma$ and $\Gamma_2 = K_n$ and noting that $P_{K_n}(x) = (n-x)^{n-1}$.

\section{Notation}
\label{sec:notation}

All graphs are assumed to be non-empty, finite, and connected. Given a graph $\Gamma$, we denote by $V(\Gamma)$ and $E(\Gamma)$, respectively, the vertex and edge set of $\Gamma$. 

We denote by $\Div(\Gamma)$ the free abelian group on $V(\Gamma)$, and refer to $D \in \Div(\Gamma)$ as \defi{divisors} on $\Gamma$. The \defi{degree} map $\deg \colon \Div(\Gamma) \to \Z$ given by 
\[
\sum a_v v \mapsto \sum a_v
\]
is a group homomorphism, and we denote the kernel of this map by $\Div^0(\Gamma)$.

An ordering $V(\Gamma) = \{v_1,\ldots,v_n\}$ of the vertices of $\Gamma$ determines a basis for $\Div(\Gamma)$. We define the \defi{Laplacian operator} $\Delta(\Gamma) \colon \Div(\Gamma) \to \Div^0(\Gamma)$ on a basis via the formula
\[
v \mapsto \left(\deg v\right) v -  \sum_{wv \in E(\Gamma)} w.
\]
Given an ordering of $V(\Gamma)$, we define the \defi{Laplacian matrix} $L(\Gamma)$ to be the matrix of $\Delta(\Gamma)$ with respect to the associated basis; $L(\Gamma)$ is equal to $D(\Gamma) - A(\Gamma)$, where $D(\Gamma)$ and $A(\Gamma)$ are, respectively, the degree and adjacency matrices of $\Gamma$. The matrix $L(\Gamma)$ is symmetric (and in particular diagonalizable with real eigenvalues) and has rank $n-c$, where $c$ is the number of connected components of $\Gamma$. If $\Gamma$ is connected, the kernel of $L(\Gamma)$ is spanned by the vector $(1, \ldots,1)$.

We define the \defi{chip-firing group} $\Pic^0(\Gamma)$ of $\Gamma$ to be $\Div^0(\Gamma) / \im \Delta(\Gamma)$; this is also frequently called the Jacobian of $\Gamma$ and denoted by $\Jac(\Gamma)$ (and also often called the critical group, or the sandpile group). As mentioned above, it has order equal to the number of spanning trees of $\Gamma$, and by Kirchhoff's matrix-tree theorem, this is equal to the product of the non-zero eigenvalues of $L(\Gamma)$ divided by the number of vertices \cite[Corollary 9.10(a)]{stanleyAlgebraicComb}. (Some authors take the dual point of view and define $\calM(\Gamma) \coloneqq \Hom(V(\Gamma),\Z)$ and the Laplacian as an operator $\calM(\Gamma) \to \Div(\Gamma)$; in our arguments, we immediately identify $\calM(\Gamma)$ and $\Div(\Gamma)$ anyway, so we skip directly to our convention.)

\begin{remark}
The name chip-firing group is closely related to \textit{chip-firing games} or \textit{dollar games} \cite{bjornerChipFiring}. 
Given a divisor $D\in \Div(\Gamma)$, one can think of an integer $a_v$ as the number of \textit{dollars} assigned to each vertex $v$. A \textit{chip-firing move} consists of picking a vertex and having it either borrow one dollar from each neighbor or giving one dollar to each of its neighbors. For $D_1,\, D_2\in \Div(\Gamma)$, $D_1-D_2 \in \im \Delta(\Gamma)$ if and only if starting from the configuration $D_1$, one can reach the configuration $D_2$ through a sequence of chip-firing moves (cf.~\cite[Lemma 4.3]{BakerN:RR}). 
This result illustrates the relationship between chip-firing moves and the chip-firing group.
\end{remark}

\begin{remark}
There is an equivalent presentation of $\Pic^0(\Gamma)$ which is useful for computing small examples, see \ref{fig:example}. Choose an orientation for each edge of $\Gamma$. Then the group $\Pic^0(\Gamma)$ is isomorphic to the abelian group generated by the oriented edges, with two types of relations. The loop relations impose that the sum of edges which form an oriented loop is trivial in the group. The vertex relations impose that for each vertex, the sum of incoming edges equals the sum of outgoing edges. 

The isomorphism to $\Pic^0(\Gamma)$ is given by sending an oriented edge to the divisor which is $1$ on the tip of the edge $-1$ on the tail, and $0$ at all other vertices. The loop relations correspond to combinations of edges which are trivial in $\Div^0(\Gamma)$, and each vertex relation corresponds to the chip-firing  at the corresponding vertex. 
We refer the reader to \cite[p.~283]{BoschLR:Neron} for more details.
\end{remark}

Let $\Div^0_{\Q}(\Gamma) \coloneqq \Div^0(\Gamma)\otimes_{\Z}\Q$.  The Laplacian operator $\Delta(\Gamma)$ induces a linear endomorphism of $\Div^0_{\Q}(\Gamma)$ that we denote by $\Delta(\Gamma)_{\Q}$. We let $P_{\Gamma}(x)$  be the characteristic polynomial of $\Delta(\Gamma)_{\Q}$; $P_{\Gamma}(x)$ has integer coefficients and degree $n-1$. We adopt the convention that $P_{\Gamma}(x)=1$ if $\Gamma$ has only one vertex.

\begin{definition}
Given a graph $\Gamma$, we say that a subset $S=\{w_1, \ldots, w_m\}\subset V(\Gamma)$ has the \defi{conformity property} if the induced subgraph on $S$ is either completely disconnected or complete, and for every vertex $x$ outside $S$, $w_ix$ is an edge if and only if $w_jx$ is an edge, for all $i,j$. 
\end{definition}

\subsection{Sketch of proof}
As stated above, $\Pic^0(\Gamma_n)$ has order equal to the determinant of $\Delta(\Gamma_n)$. 
To prove Theorem \ref{thm:cone}, we not only compute this determinant, but we also determine all of the eigenvalues of $\Delta(\Gamma_n)$. More precisely, we first isolate eigenvectors $v_{k+2}-v_{k+1},\ldots , v_{n+k} - v_{k+1}$ of $\Delta(\Gamma_n)$ that come from $K_n$, and then use the conformity property (cf.~Lemmas \ref{elemorder} and \ref{conform}) to prove that these eigenvectors generate a subgroup of $\Pic^0(\Gamma_n)$ isomorphic to $(\mathbb{Z}/(n+k)\mathbb{Z})^{n-1}$. To conclude, we show that the remaining eigenvectors of $\Delta(\Gamma_n)$ come from eigenvectors of $\Delta(\Gamma)$.

\section{Proofs}
\label{sec:main-theorem}

We begin with a pair of computational lemmas that facilitate our proofs.

\begin{lemma}\label{elemorder}
Assume $\Gamma$ is connected with at least $3$ vertices. Suppose $v_1, v_2$ are a pair of vertices of degree $d$ with the conformity property. Let $e_{12}=v_1-v_2$ as an element of $\Pic^0(\Gamma)$. Then if $v_1v_2$ is an edge, $e_{12}$ has order $d+1$, and otherwise $e_{12}$ has order $d$.
\end{lemma}

\begin{proof}

Define
\[
 \mu \coloneqq 
 \left\lbrace
\begin{array}{ll}
(d+1)e_{12} & \text{ if } v_1v_2  \text{ is an edge, and}  \\ 
de_{12} & \text{ otherwise}.
\end{array}\right.
\]
By the conformity condition, $\Delta(\Gamma)(v_1 - v_2)= \mu$. It remains to show that no smaller multiple of $e_{12}$ is trivial in $\Pic^0(\Gamma)$. 
Suppose $m$ is an integer such that $me_{12}$ is trivial in $\Pic^0(\Gamma)$. Then there exists $D = \sum a_i v_i \in \Div \Gamma$ such that  $\Delta(\Gamma)(D) = me_{12}$. Since
\[
m\Delta(\Gamma)(v_1 - v_2)=  \mu\Delta(\Gamma)(D)
\] 
in $\Pic^0(\Gamma)$, and since the kernel of $\Delta(\Gamma)_{\Q}$ is spanned by $\sum  v_i$, there exists a rational number $r \in \Q$ such that 
\[
\mu\sum a_i v_i  -  m (v_1 - v_2) = r \sum  v_i.
\]
In particular $\mu a_1=r+m, \mu a_2=r-m$, and $\mu a_i=r$ for $i>2$. But then, since $\mu$, $m$, and each $a_i$ are integers,  $r$ is also an integer. 
Subtracting $r \sum  v_i$ from both sides then gives 
\[
\mu(a_1' v_1 - a_2' v_2)  =  m (v_1 - v_2)
\]
where $a_i' = a_i -r/\mu \in \Z$. But then $\mu a_i' = m$, and we conclude that $\mu$ divides $m$.
\end{proof}

Next, we generalize Lemma \ref{elemorder} to proper subgraphs with the conformity property.

\begin{lemma}\label{conform}
Let $j\geq 1$ and let  $S^1=\{v^1_{1}, \ldots, v^1_{m_1}\}, \ldots, S^j=\{ v^j_{1}, \ldots, v^j_{m_j} \}$ be $j$ mutually disjoint vertex sets, each with the conformity property. Assume $\Gamma$ is connected and the sets $S^1, \ldots, S^j$ do not completely cover $\Gamma$. Then if the elements $e^i_{1k}=v^i_{1}-v^i_{k}$, where $i$ ranges from $1$ to $j$, and $k$ ranges from $2$ to $m_i$, satisfy a relation $\sum_{i=1}^j\sum_{k=2}^{m_i} \alpha_{ij} e^i_{1k}=0$ in $\Pic^0(\Gamma)$, then each $\alpha_{ij}e^i_{1k}=0$ in $\Pic^0(\Gamma)$ .
\end{lemma}

The order of $e^i_{1k}$ is specified by the previous lemma (and only depends on $i$).

\begin{proof}
Let $\mu^i_k$ be the order of the element $e^i_{1k}$ in $\Pic^0(\Gamma)$. If $n^i_k$ are integers such that 
\[
 \sum_{i=1}^{j} \sum_{k=2}^{m_i} n^i_k e^i_{1k} = 0 \text{ in } \Pic^0(\Gamma),
\]
then there exists $D \in \Div \Gamma$ such that 
\[
\Delta(\Gamma)(D) =  \sum_{i=1}^{j} \sum_{k=2}^{m_i} n^i_k e^i_{1k}.
\]
But for each $i$ and $k$, Lemma \ref{elemorder} asserts that
\[
\Delta(\Gamma)(v^i_1 - v^i_k) = \mu^i_k e^i_{1k},
\]
and thus
\[
\sum_{i=1}^{j} \sum_{k=2}^{m_i} n^i_k e^i_{1k}  = \Delta(\Gamma)\left(\sum_{i=1}^{j} \sum_{k=2}^{m_i} \frac{n^i_k}{\mu^i_k}(v^i_1-v^i_k)\right).
\]
Since the kernel of $\Delta(\Gamma)$ is spanned by $\sum_{v \in V(\Gamma)} v$, there thus exists a rational number $r$ such that 
\[
D- \sum_{i=1}^{j} \sum_{k=2}^{m_i}\frac{n^i_k}{\mu^i_k}(v^i_1-v^i_k) = r \sum_{v \in V(\Gamma)} v.
\]
Since the vertex sets $S^i$ do not completely cover $\Gamma$, there is at least one vertex $v$ not included among the $v^i_k$; comparing the coefficient of $v$, we conclude that $r$ is an integer (since the coefficient of $v$ in $D$ is an integer). Thus each fraction $\frac{n^i_k}{\mu^i_k}$ must also be an integer, so each $n^i_k$ is divisible by $\mu^i_k$.
\end{proof}

We now prove our main theorems.

\begin{proof}[Proof of Theorem \ref{thm:cone}]

When $k = 1$, $\Gamma_n = K_{n+1}$, and it is well known that $\Pic^0 (K_{n+1})$ is isomorphic to $(\Z / (n-1)\Z)^{n+1}$, in which case we directly observe that the theorem is true.    Thus we may assume that $k\geq 2$.

Consider the matrix $B_n=L(\Gamma_n)-L(K_{n+k})$. Every entry is $0$ except for the upper $k$ by $k$ submatrix, which we denote by $B_0$. Now, $B_0=L(\Gamma)-L(K_{k})$.

Now, note that $B_n$ acts on $\Div^0_{\Q}(\Gamma_n)$. Let $\mathbf{u} \in \Div^0_{\Q}(\Gamma_n)$ be an eigenvector of $B_n$ with eigenvalue $\mu$. As $K_{n+k}$ is a complete graph, the matrix $L(K_{n+k})$ acts as multiplication by $n+k$ on all elements of $\Div^0_{\Q}(\Gamma_n)$. Thus $\mathbf{u}$ is an eigenvector of the operator $\Delta(K_{n+k})$ with eigenvalue $n+k$, so it is an eigenvalue of $\Delta(\Gamma_n)$ with eigenvalue $n+k+\mu$.


Choose $k-1$ eigenvectors $\mathbf{u}_1, \ldots, \mathbf{u}_{k-1} \in \Div^0_{\Q}(\Gamma)$ of $B_0$, with corresponding eigenvalues $\mu_i$. Then, by appending $n$ zeros to each vector, we get eigenvectors of $\Delta(\Gamma_n)$ with eigenvalues $n+k+\mu_i$. On the other hand, $\mathbf{u}_1, \ldots, \mathbf{u}_{k-1}$ are eigenvectors of $\Delta(\Gamma)$ with eigenvalues $k+\mu_i$.

For $i>k+1$, the $n-1$ vectors $\mathbf{u}_i = v_i-v_{k+1}$ are eigenvectors of $\Delta(\Gamma_n)$ with eigenvalue $n+k$. Finally, the vector 
\[
n\sum_{i = 1}^{k}  v_i - k\sum_{i =k+1}^{n+k}  v_i
\] 
is an eigenvector, also with eigenvalue $n+k$. We have given a basis for $\Div^0_{\Q}(\Gamma_n)$ in eigenvectors of $\Delta(\Gamma_n)$.

By the matrix-tree theorem, the order of $\Pic^0(\Gamma_n)$ is the product of these eigenvalues, divided by $n+k$, which is $(n+k)^{n-1}\prod (n+k+\mu_i)$.

Now, the elements $v_{k+1}-v_{k+1+i}$ for $i>0$ generate a subgroup isomorphic to $(\mathbb{Z}/(n+k)\mathbb{Z})^{n-1}$ by Lemma \ref{conform}. The quotient has order $\prod (n+k+\mu_i)$. But the $k+\mu_i$ are the eigenvalues of $\Delta(\Gamma)$ acting on $\Div^0_{\Q}(\Gamma)$, so this expression is equal to $|P_{\Gamma}(-n)|$.
\end{proof}


\begin{proof}[Proof of Theorem \ref{thm:tree}]
Let $v_1, \ldots, v_k$ be the vertices of $\Gamma$. Let $\Gamma_n$ be the join of $\Gamma$ with the complete graph $K_n$ on the vertices $w_1, \ldots, w_n$. The group $H_n$ is the quotient of $\Pic^0(\Gamma_n)$ by relations generated by $w_i-w_j$; in the following, for an element $w \in \Pic^0(\Gamma_n)$ we write $\bar{w}$ for the image of $w$ in $H_n$.
Thus $H_n$ is generated by the elements $\bar{v}_i-\bar{w}_1$. Suppose $v_1, \ldots, v_l$ are leaves of $\Gamma$, and let $H'$ be the subgroup of $H_n$ generated by $\bar{v}_j-\bar{w}_1$, for $1\leq j \leq l$. We will show that $H'=H_n$.

Let $S$ be the set of vertices $v_i$ of $\Gamma$ such that for every vertex $v' \in \Gamma_n$ connected to $v_i$, $\bar{v}_i-\bar{v}' \in H'$. By construction, the vertices $v_1, \ldots, v_l$ are in $S$. 
Now suppose $v$ is a vertex of $\Gamma$ with the following property: at least one neighbor of $v$ in $\Gamma$ is in $S$ and at most one neighbor of $v$ is not in $S$.  Let $\gamma_1 \ldots \gamma_j$ be the neighbors of $v$ in $\Gamma$ which are in $S$. Then $\bar{\gamma_1}-\bar{w}_1$ and $\bar{\gamma_1}-\bar{v}$ are in $H'$, so $\bar{v}-\bar{w}_1$ is as well. Thus $\bar{v}-\bar{w}_i$ is as well, for any $w_i \in K_n$, so we have accounted for every edge of $v$ except possibly one, call the other vertex $\delta$. By firing $v$ we see that $\sum (v-w_i) + \sum (v-\gamma_i) + (v-\delta)$ is trivial in $\Pic^0(\Gamma_n)$. Thus the image is trivial in $H'$, and as every element of this sum is an element of $H'$ except perhaps $v-\delta$, we must have $v-\delta$ is in $H'$. We conclude that $v \in S$.

Let $S^c$ be the complement of $S$ in $\Gamma$. If $S^c$ were nonempty, the induced subgraph $\Gamma_{S^c}$ of $\Gamma$ on $S^c$ would be a tree, so it must have a leaf $v$. The leaves of $\Gamma$ are in $S$, so $v$ is not a leaf of $\Gamma$. So at least one neighbor of $v$ in $\Gamma$ is in $S$ and at most one neighbor is not in $S$. We conclude that $v \in S$, and therefore every vertex of $\Gamma$ is in $S$.
\end{proof}


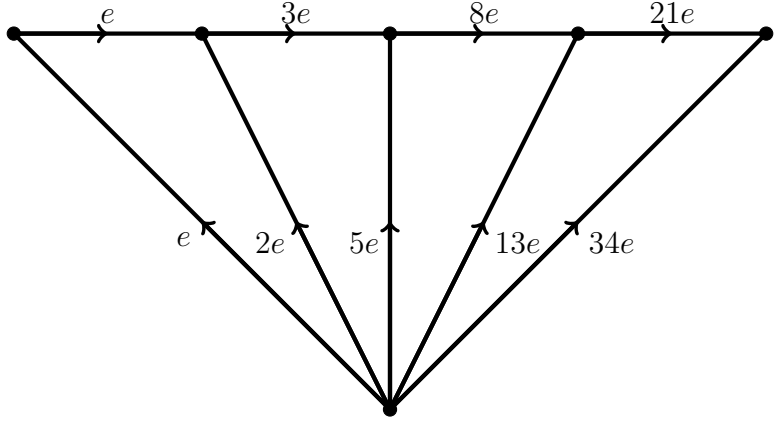
\begin{figure}[h!]
\begin{tikzpicture}[scale=2.5][line join=bevel]
\coordinate (0) at (-2,0);
\coordinate (1) at (-1,0);
\coordinate (2) at (0,0);
\coordinate (3) at (1,0);
\coordinate (4) at (2,0);
\coordinate (5) at (0,-2);
\coordinate (6) at (-1.5,0);
\coordinate (7) at (-1/2,0);
\coordinate (8) at (1/2,0);
\coordinate (9) at (3/2,0);
\coordinate (10) at (-2/2,-2/2);
\coordinate (11) at (-1/2,-2/2);
\coordinate (12) at (0,-2/2);
\coordinate (13) at (1/2,-2/2);
\coordinate (14) at (2/2,-2/2);

\filldraw [black] (0) circle (1pt) 
(1) circle (1pt)
(2) circle (1pt)
(3) circle (1pt)
(4) circle (1pt)
(5) circle (1pt);


\draw [ultra thick] (0)--(1);
\draw [ultra thick] (1)--(2);
\draw [ultra thick] (2)--(3);
\draw [ultra thick] (3)--(4);
\draw [ultra thick] (5)--(0);
\draw [ultra thick] (5)--(1);
\draw [ultra thick] (5)--(2);
\draw [ultra thick] (5)--(3);
\draw [ultra thick] (5)--(4);

\draw [ultra thick, ->] (0)--(6);
\draw [ultra thick, ->] (1)--(7);
\draw [ultra thick, ->] (2)--(8);
\draw [ultra thick, ->] (3)--(9);

\draw [ultra thick, ->] (5)--(10);
\draw [ultra thick, ->] (5)--(11);
\draw [ultra thick, ->] (5)--(12);
\draw [ultra thick, ->] (5)--(13);
\draw [ultra thick, ->] (5)--(14);

\draw (6) node [anchor=south] {$e$};
\draw (7) node [anchor=south] {$3e$};
\draw (8) node [anchor=south] {$8e$};
\draw (9) node [anchor=south] {$21e$};
\draw (10) node [anchor=north east] {$e$};
\draw (11) node [anchor=north east] {$2e$};
\draw (12) node [anchor=north east] {$5e$};
\draw (13) node [anchor=north west] {$13e$};
\draw (14) node [anchor=north west] {$34e$};
\end{tikzpicture}

\caption{Theorem \ref{thm:tree} in action: group relations for the cone $\Gamma_1$ over the path graph $\Gamma$ on $5$ vertices. We orient the edges to give a presentation for $\Pic^0 (\Gamma_1)$, such that each oriented edge represents the function which is $1$ at the tip, $-1$ at the tail, and $0$ at all other vertices. The signed sum of the edges around every loop is trivial in $\Pic^0 (\Gamma_1)$, and the chip-firing relations impose that the signed sum of the edges incident to each vertex is trivial as well. Firing at the cone vertex shows that $55e$ is trivial in the group, and $\Pic^0 (\Gamma_1) \cong \mathbb{Z}/55\mathbb{Z}$.} \label{fig:example}
\end{figure}


\begin{proof}[Proof of Theorem \ref{thm:bipartite}]
Let $B_i$ be the matrix $\Delta(\Gamma_i)-\Delta(K_{k_i})$. Then $B_i$ acts on  $\Div^0_{\Q}(\Gamma_i)$, which admits a basis of eigenvectors $\mathbf{u}^i_{1}, \ldots, \mathbf{u}^i_{k_i-1}$, with eigenvalues $\mu^i_{j}$. These eigenvectors are also eigenvectors of $\Delta(\Gamma_i)$, with eigenvalues $\mu^i_{j}+k_i$.

Now let $B_J$ be the matrix $L(\Gamma_J)-L(K_{k})$. Then $B_J$ is block diagonal, where the $i$th block is a copy of $B_i$. The eigenvalues of $B_J$ acting on $\Div^0_{\Q}(\Gamma_J)$  thus include the eigenvalues of each $B_i$. The remaining eigenvalues are all $0$, since the sum of the rows of any $B_i$ is $0$.

Thus the eigenvalues of $\Delta(\Gamma_J)$ acting on $\Div^0_{\Q}(\Gamma_J)$ are $\mu^i_j+k$, along with $l-1$ copies of $k$. By the matrix-tree theorem, the order of $\Pic^0(\Gamma_J)$ is $k^{l-2}\prod_{j} \prod_{i} (\mu^i_j+k)$, and note that $\prod_{i} (\mu^i_j + k)=|P_{\Gamma_i}(k_i-k)|$.
\end{proof}
\section*{Acknowledgments}
The authors thank the referee for their detailed and thoughtful comments.

\bibliographystyle{plain}
\bibliography{master.bbl}

\end{document}